\documentclass[11pt,reqno,tbtags]{amsart}
\usepackage{amssymb}
\usepackage{natbib}
\bibpunct[, ]{[}{]}{;}{n}{,}{,}

\numberwithin{equation}{section}

\allowdisplaybreaks

\newtheorem{theorem}{Theorem}[section]
\newtheorem{lemma}[theorem]{Lemma}

\newtheorem{corollary}[theorem]{Corollary}
\newtheorem{condition}[theorem]{Condition}

\theoremstyle{definition}
\newtheorem{example}[theorem]{Example}

\newtheorem{remark}[theorem]{Remark}

\newtheorem*{ack}{Acknowledgement}

\theoremstyle{remark}

\newenvironment{romenumerate}[1][0pt]{
\addtolength{\leftmargini}{#1}\begin{enumerate}
 }{\end{enumerate}}

\newcounter{oldenumi}
{\setcounter{oldenumi}{\value{enumi}}
\begin{romenumerate} \setcounter{enumi}{\value{oldenumi}}}
{\end{romenumerate}}

\newcounter{thmenumerate}

\newcounter{romxenumerate}   

\newcounter{xenumerate}   

\newcommand{\refT}[1]{Theorem~\ref{#1}}
\newcommand{\refC}[1]{Corollary~\ref{#1}}
\newcommand{\refL}[1]{Lemma~\ref{#1}}
\newcommand{\refR}[1]{Remark~\ref{#1}}
\newcommand{\refS}[1]{Section~\ref{#1}}

\newcommand{\refE}[1]{Example~\ref{#1}}

\newcommand{\refand}[2]{\ref{#1} and~\ref{#2}}

\newcommand{\refCN}[1]{Condition~\ref{#1}}
\newcommand\refCNN{\refCN{C1}}
\newcommand\refCC[1]{\refCN{C1}\ref{#1}}
\newcommand{\refCNB}{Conditions \ref{C1} and~\ref{C2}}

\newcommand\marginal[1]{\marginpar{\raggedright\parindent=0pt\tiny #1}}

\begingroup
  \divide\count255 by 60
  \multiply\count255 by -60
  \advance\count255 by \time
  \ifnum \count255 < 10 \xdef\klockan{\the\count1.0\the\count255}
  \else\xdef\klockan{\the\count1.\the\count255}\fi
\endgroup

\DeclareMathOperator*{\sumx}{\sum\nolimits^{*}}

\newcommand{\sumin}{\sum_{i=1}^n}

\newcommand{\sumk}{\sum_{k=0}^\infty}
\newcommand{\sumki}{\sum_{k=1}^\infty}

\newcommand\set[1]{\ensuremath{\{#1\}}}

\newcommand\xpar[1]{(#1)}
\newcommand\bigpar[1]{\bigl(#1\bigr)}
\newcommand\Bigpar[1]{\Bigl(#1\Bigr)}

\newcommand\lrpar[1]{\left(#1\right)}
\newcommand\bigsqpar[1]{\bigl[#1\bigr]}

\newcommand\abs[1]{|#1|}

\newcommand\Bigabs[1]{\Bigl|#1\Bigr|}

\def\rompar(#1){\textup(#1\textup)}    
\newcommand\xfrac[2]{#1/#2}

\newcommand\Bigparfrac[2]{\Bigpar{\frac{#1}{#2}}}

\def\xexp(#1){e^{#1}}

\newcommand\floor[1]{\lfloor#1\rfloor}

\newcommand\ntoo{\ensuremath{{n\to\infty}}}

\newcommand\epsto{\ensuremath{{\eps\to0}}}
\newcommand\bmin{\wedge}

\newcommand\downto{\searrow}
\newcommand\upto{\nearrow}

\newcommand\punkt[1]{\if.#1\else.\spacefactor1000\fi{#1}}

\newcommand\ie{i.e\punkt}
\newcommand\eg{e.g\punkt}

\newcommand\cf{cf\punkt}
\newcommand{\as}{a.s\punkt}

\newcommand\whp{{w.h.p.\spacefactor=1000}}
\newcommand\hp{\widehat p}
\newcommand\hg{\widehat g}
\newcommand\hD{\widehat D_\infty}

\newcommand{\tend}{\longrightarrow}
\newcommand\dto{\overset{\mathrm{d}}{\tend}}
\newcommand\pto{\overset{\mathrm{p}}{\tend}}
\newcommand\asto{\overset{\mathrm{a.s.}}{\tend}}
\newcommand\lito{\overset{{L^ 1}}{\tend}}
\newcommand\eqd{\overset{\mathrm{d}}{=}}

\newcommand\op{o_{\mathrm p}}

\newcommand\bbN{\mathbb N}  

\newcommand\bbZ{\mathbb Z}

\newcounter{CC}
\newcommand{\CC}{\stepcounter{CC}\CCx} 
\newcommand{\CCx}{c_{\arabic{CC}}}     
\newcommand\cc[1]{\mathcal C_{#1}}

\newcommand\E{\operatorname{\mathbb E{}}}
\renewcommand\P{\operatorname{\mathbb P{}}}
\newcommand\Var{\operatorname{Var}}

\newcommand\ga{\alpha}

\newcommand\gl{\lambda}

\newcommand\eps{\varepsilon}

\newcommand\cA{\mathcal A}

\newcommand\cC{\mathcal C}

\newcommand\cE{\mathcal E}

\newcommand\ett[1]{\boldsymbol1[#1]} 
\newcommand\etta{\boldsymbol1} 
\def\[#1]{[\![#1]\!]}

\newcommand\qw{^{-1}}
\newcommand\qww{^{-2}}

\renewcommand{\=}{:=}

\newcommand\intoi{\int_0^1}
\newcommand\intoe{\int_0^\eps}

\newcommand\dtv{d_{\mathrm{TV}}}

\newcommand\ddd{\,\textup{d}}

\newcommand{\pgf}{probability generating function}

\newcommand\rhs{right-hand side}

\newcommand\gnp{\ensuremath{G(N,p)}}
\newcommand\gnm{\ensuremath{G(n,m)}}
\newcommand\gnd{\ensuremath{G(n,d)}}
\newcommand\gnx[1]{\ensuremath{G(n,#1)}}
\newcommand\gndd{\gnx{\dd}}
\newcommand\gnddx{\ensuremath{G^*(n,\dd)}}
\newcommand\gndl{\ensuremath{G(n_\gl,\dd_\gl)}}

\newcommand\dn{\ensuremath{(d_i)_1^n}}
\newcommand\din{\ensuremath{(d_i)_{i=1}^n}}
\newcommand\dd{\ensuremath{\mathbf d}}
\newcommand\ndd{(n,\dd)}

\newcommand\sus{\chi}
\newcommand\susq{\widehat\chi}
\newcommand\susoo{\chi_\infty}
\newcommand\susqoo{\widehat\chi_\infty}

\newcommand\cci{\cc{i}}

\newcommand\ccc[1]{|\cc{#1}|}
\newcommand\ccci{\ccc{i}}

\newcommand{\restr}[1]{|_{#1}}
\newcommand\sumiK{\sum_{i=1}^K}
\newcommand\sumixK{\sum_{i=2}^K}
\newcommand\bp{\ensuremath{\mathfrak X}}
\newcommand\hbp{\ensuremath{\hat{\mathfrak X}}}

\newcommand\kk{\varkappa}
\newcommand\qbp{|\bp|}
\newcommand\qhbp{|\hbp|}
\newcommand\nx[1]{\ensuremath{N_{#1}}}
\newcommand\nk{\nx{k}}

\newcommand\susbp{\sus(\bp)}
\newcommand\susqbp{\susq(\bp)}
\newcommand\rhox[1]{\ensuremath{\rho_{#1}}}
\newcommand\rhok{\rhox{k}}

\newcommand\rhooo{\rhox{\infty}}

\newcommand\muq{\widehat\mu}
\newcommand\nuq{\widehat\nu}

\newcommand\nn{^{(n)}}

\newcommand\elli{{\ell-1}}
\newcommand\ellii{{\ell-2}}
\newcommand\xn{[n]}
\newcommand\co{^\circ}

\newcommand\hG{\widehat G}

\newcommand\tG{\widetilde G}
\newcommand\tGx{\widetilde G^*}
\newcommand\hex[1]{\widehat\cE_{#1}}
\newcommand\hea{\widehat\cE_A}
\newcommand\heac{\widehat\cE_A^c}
\newcommand\cea{\cE_A}
\newcommand\ceb{\cE_B}
\newcommand\xio{\xi_0}
\newcommand\hxio{\widehat\xi_0}
\newcommand\hxi{\widehat\xi}
\newcommand\dx{\ensuremath{D^*_\infty}}
\newcommand\gx{g_*}
\newcommand\vn{\xn}
\newcommand\vna{\vn\setminus A}
\newcommand\tn{\widetilde n}
\newcommand\gdd{\delta}
\newcommand\ggl{g_\gl}
\newcommand\gglc{g_{\glc}}
\newcommand\glc{\gl_\textsf{c}}
\newcommand\muc{\mu_\textsf{c}}
\newcommand\loge{\log\eps}
\newcommand\aloge{\log(1/\eps)}

\newcommand{\sumko}{\sum_{1\le k<\infty}}
\newcommand{\sumkoo}{\sum_{1\le k\le\infty}}
\newcommand{\muoo}{\mu_\infty}
\newcommand{\nuoo}{\nu_\infty}
\newcommand{\muqoo}{\muq_\infty}
\newcommand{\nuqoo}{\nuq_\infty}
\newcommand{\mun}{\mu_n}
\newcommand{\nun}{\nu_n}
\newcommand{\hmuoo}{\widehat\mu_\infty}
\newcommand{\hnuoo}{\widehat\nu_\infty}
\newcommand{\doo}{D_\infty}
\newcommand{\Doo}{D_\infty}

\newcommand\opn{\op(n)}

\newcommand\ER{Erd\H os--R\'enyi}

\newcommand\REM[1]{{\raggedright\texttt{[#1]}\par\marginal{XXX}}}

\newcommand\citetq[2]{\citeauthor{#2} \cite[{\frenchspacing #1}]{#2}}

\newcommand\urladdrx[1]{{\urladdr{\def~{{\tiny$\sim$}}#1}}}

\begin{document}
\title
{Susceptibility of random graphs with given vertex degrees}

\date{November 13, 2009} 

\author{Svante Janson}
\address{Department of Mathematics, Uppsala University, PO Box 480,
SE-751~06 Uppsala, Sweden}
\email{svante.janson@math.uu.se}
\urladdrx{http://www.math.uu.se/~svante/}


\begin{abstract} 
We study the susceptibility, \ie, the mean cluster size, in random
graphs with given vertex degrees. We show, under weak assumptions,
that the susceptibility converges to the expected cluster size in the
corresponding branching process. In the supercritical case, a
corresponding result holds for the modified susceptibility ignoring
the giant component and the expected size of a finite cluster in the
branching process; this is proved using a duality theorem.

The critical behaviour is studied. Examples are given where the
critical exponents differ on the subcritical and supercritical sides.
\end{abstract}

\maketitle

\section{Introduction}\label{S:intro}

The \emph{susceptibility} $\sus(G)$ of a graph $G$ 
is defined as the mean size of the component containing a random
vertex:
\begin{equation}\label{sus1}
 \sus(G)\=|G|^{-1}\sum_{v\in V(G)}|\cC(v)|,
\end{equation}
where $\cC(v)$ denotes the component of $G$ containing the vertex $v$.
Thus, if $G$ has $n=|G|$ vertices and 
components $\cci=\cci(G)$, $i=1,\dots, K$, 
where $K=K(G)$ is the number of
components, then
\begin{equation}\label{sus}
  \sus(G)\=\sumiK \frac{\ccci}{n}\ccci
=\frac1n\sumiK \ccci^2.
\end{equation}
Although it does not matter here, we assume for later use that the components
as usual are ordered with  $\ccc1\ge\ccc2\ge\cdots$.

When the graph $G$ is itself 
random, $\sus(G)$ is thus a random variable.
(We do not take the expectation over $G$ unless we explicitly write
$\E\sus(G)$.)

The susceptibility (in particular its expectation)
has been much studied for certain models in
mathematical physics. 
(That is the reason for using the term susceptibility, and the
notation $\sus$, which both come from physics.) 
Similarly, in percolation theory, which deals
with certain random infinite graphs, the corresponding quantity is
the (mean) size of the open cluster containing a given vertex, and this
has been extensively studied; see \eg{} \citet{BRbook}.
For finite random graphs, there are some papers:
\citet{SW} studied in a pioneering paper a class of random graph processes
(including the \ER{} graph process)
and used the susceptibility to study the phase transition in them.
Some results for the \ER{} random graphs \gnp{} and \gnm{} can be
regarded as folk theorems; detailed results are given by
\citetq{Section 2.2}{Durrett} and  \citet{SJ218}.
\citet{BCHSS} give precise results for random subgraphs of
transitive graphs (including both $\gnp$ and, for example, random
subgraphs of the hypercube); further results for
random subgraphs of the hypercube are given by \citet{vdHS05,vdHS06}.
A class of inhomogeneous random graphs is studied by \citet{SJ232},
see also \citet{ChayesSmith}.
Another application is given in \citet{sjspencer}. We refer to these
papers for further background.
The purpose of the present paper is to study the susceptibility for
the random graph $\gndd$ with given vertex degrees, where $\dd=\dn$ is
a given degree sequence (see \refS{Sprel} for a detailed definition).
This case has earlier been studied in a heuristic way by
\citet{NSW01}, using the branching process in \refS{Sbranch} below.

The definition \eqref{sus} is mainly interesting in the subcritical
case, when all components are rather small. In the supercritical case,
see \citet{MR95} or \refT{Tgiant} below,
there is one giant component that is so large that it
dominates the sum in \eqref{sus};
in fact, for some $\rho>0$, $\ccc1=(\rho+\op(1))n$ while 
$\ccc2=\opn$ and thus
\begin{equation*}
  \sumiK\ccci^2
=\ccc1^2+O\Bigpar{\ccc2\sumixK\ccci}
=(\rho^2+\op(1))n^2
=(1+\op(1))\ccc1^2.
\end{equation*}
It then makes sense to exclude the largest component from the
definition, 
and we define as in \cite{SJ232}
the {\em modified susceptibility} $\susq(G)$
of a finite graph
$G$ by
\begin{equation}\label{susq}
  \susq(G)\=\frac1n\sumixK \ccci^2.
\end{equation}
(This is in analogy with percolation theory, where one
studies the mean size of the open cluster containing a given vertex,
given that this cluster is finite.)

Our main result is the following, 
giving the asymptotics of both $\sus$
and $\susq$ for $\gndd$,
using notation introduced in \refS{Sprel} below, 

\begin{theorem}
  \label{Tsus}
Suppose that Conditions \refand{C1}{C2} hold.
\begin{romenumerate}
\item
In the subcritical case $\nuoo<\muoo$, 
\begin{align*}
  \sus(\gndd)&\lito\susoo\=1+\frac{\muoo^2}{\muoo-\nuoo},
\\
  \susq(\gndd)&\lito\susoo.
\end{align*}
\item
In the critical case $\nuoo=\muoo$, 
\begin{align*}
  \sus(\gndd),  
\susq(\gndd)&\pto\infty.
\end{align*}
\item
In the supercritical case $\nuoo>\muoo$, 
with $\kk\in(0,1)$ given by $g'(\kk)=\kk g'(1)$,
\begin{align*}
  \sus(\gndd)&\pto\infty,
\\
  \susq(\gndd)&\pto\susqoo
\=
g(\kk)+\frac{\kk(g'(\kk))^ 2}{g'(\kk)-\kk g''(\kk)}
<\infty.
\end{align*}
\end{romenumerate}
\end{theorem}

Note that the $L^ 1$-convergence in the subcritical case (i) entails
both $\sus(\gndd)\pto\susoo$ and $\E\sus(\gndd)\to\susoo$.
Hence, in all three cases we have
\begin{align*}
  \sus(\gndd)&\pto\susoo\=1+\frac{\muoo^2}{(\muoo-\nuoo)_+} \le\infty,
\\
  \susq(\gndd)&\pto\susqoo
\=
g(\kk)\Bigpar{1+\frac{\muqoo^2}{(\muqoo-\nuqoo)_+}}
\le\infty.
\end{align*}
Further, in (ii) and (iii) it follows trivially that
$\E\sus(\gndd)\to\infty$. Hence also $\E\sus(\gndd)\to\susoo$ holds in
all three cases. However, our proof does not (at least not
immediately) show convergence of $\E \susq(\gndd)$ in the supercritical case
(iii), although we conjecture that it holds there too.

The results are based on
approximation by a branching process $\bp$, see \refS{Sbranch},
as is standard when studying the component structure in both \gndd{}
and in several other random graph models (see \eg{} \cite{JLR} and
\cite{kernels}). 
\refT{Tsus} can be seen as saying that (under some
weak conditions), the 
susceptibility $\sus$ and the modified susceptibility $\susq$ of
$\gndd$ converge to 
the corresponding mean values for the branching process
corresponding to \gndd; see \refT{Tsus2} for details.
Proofs are given in Sections \ref{Slow}--\ref{Spf}.

The proof of our result for $\susq$ is based on a duality result, 
\refT{Tdual},
saying  that if we delete the largest component $\cc1$ from a supercritical
\gndd, then the remainder is essentially another random graph of the
same type, which furthermore is subcritical. 
(The size and vertex degrees are random, but this is not
important since they are  concentrated.)
This was proved already by \citet{MR98}, but we need a slightly
sharper form here.
Such duality results for \gnp{} go back to  \citet{BB84},
see also \citet{Lucz90}, \citet{SJgiant} and the books \citet{BBbook},
\citet{JLR}; a generalization to a class of inhomogeneous random
graphs is given by
\citetq{Theorem 12.1}{kernels} and a further generalization by 
\citet{SJcutdual}.

\refT{Tsus} is stated as a limit result. An alternative is to
formulate the result as an approximation for finite $n$; this version
is given in \refS{Sapp}.
We end with some further comments. The behaviour close to criticality
is studied in \refS{Scrit} for a specific situation. We show that
there is symmetry between the 
subcritical and supercritical sides when the asymptotic degree
distribution has a third moment, but not necessarily in general; the
critical exponent on the subcritical side is always 1 but on the
supercritical side it may be arbitrarily large.
Finally, in \refS{Sex} we give some examples
showing that the main theorems may fail without our conditions.

\begin{ack}
  Parts of this research has been done during visits to
Centre de recherches math\'ematiques, Montreal (Canada),
Institut Mittag-Leffler, Djursholm (Sweden)
and Institut Henri Poincar\'e, Paris (France).
\end{ack}

\section{Preliminaries}\label{Sprel}

Let $n \in \bbN$ and let $\dd=\din$ be a sequence of
non-negative integers. 
We let \gndd{} be a random graph with degree sequence 
$(d_i)_1^n$, uniformly chosen among all possibilities (tacitly
assuming that there is any such graph at all; in particular,
$\sum_id_i$ has to be even).

As in many papers on these random graphs, we find it convenient 
to consider the corresponding
random \emph{multigraphs} 
generated by the \emph{configuration model} 
(see 
Bollob\'as \cite{BB80} and
\cite[Section II.4]{BBbook};
see also Bender and Canfield
\cite{BenderC} and Wormald \cite{WormaldPhD,Wormald81} for related
arguments):
Let $n \in \bbN$ and let $(d_i)_1^n$ be a sequence of
non-negative integers such that  $\sumin d_i$ is even.
Then take a set of $d_i$ \emph{half-edges} for each vertex 
$i$, and combine the half-edges into pairs by a uniformly random
matching of the set of all half-edges (allowing multiple edges and loops);
this yields
the random multigraph \gnddx{} with given degree
sequence $(d_i)_1^n$.
Conditioned on the multigraph being a (simple) graph, we obtain 
\gndd,
the
uniformly distributed random graph with the given degree sequence.

We assume throughout the paper that we are given a sequence 
$(d_i)_1^n=(d_i\nn)_1^n$ 
for each $n\in\bbN$ (or at least for some sequence \ntoo); 
for notational simplicity we will usually not show the dependence on
$n$ explicitly for these and some other quantities.
We consider asymptotics as \ntoo, and all unspecified limits below are
as \ntoo.
We say that an event holds \whp{} (\emph{with high probability}), 
if it holds with probability tending to 1 as 
$n\to\infty$.
We use standard probabilistic notations for convergence; in particular
$\pto$ and $\dto$ for convergence in probability and in distribution,
and $\op$ in the standard way
(see \eg{} 
\cite{JLR} or \cite{SJN6}):
for example,
if $(X_n)$ is a sequence of random variables, then
$X_n=\op(1)$ means that $X_n \pto0$.

We write 
\begin{align*}
n_k&=n_k(n)\=\#\set{i:d_i=k},
\quad k\ge0,    
\intertext{and} 
m&=m(n)\=\tfrac12\sumin d_i
=\tfrac12\sumk k n_k;
\end{align*}
thus
$n_k$ is the number of vertices of degree $k$  
and 
$m$ is the number of edges 
in the random graph \gndd{} (or \gnddx). 
We assume as in \cite{SJ204} that the given
$(d_i)_1^n$ satisfy the following regularity conditions,
\cf{} Molloy and Reed \cite{MR95,MR98} (where similar but not
identical conditions are assumed).

\begin{condition}\label{C1}
For each $n$,  $(d_i)_1^n=(d_i\nn)_1^n$ is a sequence of non-negative
integers such that $\sumin d_i$ is even.
Furthermore,
$(p_k)_{k=0}^\infty$ is a probability distribution independent of $n$
such that
\begin{romenumerate}
  \item\label{C1p}
$n_k/n=\#\set{i:d_i=k}/n\to p_k$ as \ntoo, for every $k\ge0$;
\item \label{C1l}
$\sum_k k p_k\in(0,\infty)$;
\item \label{C1d2}
$\sum_i d_i^ 2=O(n)$;
\item \label{C1p1}
$p_1>0$.
\end{romenumerate}
\end{condition}

Let $D_n$ be a random variable defined as 
$d_I$ for a uniform random index $I\in\set{1,\dots,n}$: thus $D_n$ is
the degree of a random
(uniformly chosen) vertex in \gndd{} or \gnddx, and 
\begin{equation}
  \P(D_n=k)=n_k/n.
\end{equation}
Define
\begin{align}
  \mu_n&\=\E D_n =\frac1n\sumin d_i=\frac{2m}n,\label{mun}\\
  \nu_n&\=\E D_n(D_n-1) =\frac1n\sumin d_i(d_i-1).\label{nun}
\end{align}
Further,
let $\Doo$ be a random variable with the distribution $\P(\Doo=k)=p_k$,
and extend \eqref{mun} and \eqref{nun} to 
$\muoo\=\E D_\infty$ and $\nuoo\=\E D_\infty(D_\infty-1)$.
Then \refCC{C1p} can be written
\begin{equation}
  \label{add}
D_n\dto \Doo.
\end{equation}
In other words, $\Doo$ describes the asymptotic distribution of the degree of a
random vertex in \gndd.
Furthermore, \ref{C1l} is 
$\muoo=\E \Doo\in(0,\infty)$,
\ref{C1p1} is $\P(\Doo=1)>0$,
and
\ref{C1d2} can be written
\begin{equation}
  \label{aa}
\E D_n^2=O(1)
\end{equation}
or, equivalently, $\nu_n=O(1)$.

\begin{remark}\label{RC1}
  In particular, \eqref{aa} implies that the random variables $D_n$
  are uniformly integrable, and thus \refCN{C1}\ref{C1p}, in the form
  \eqref{add},
 implies
$\E D_n\to\E \Doo$, \ie{}
 \begin{equation}
   \label{mn}
\mun=\frac{2m}{n}=\frac1n\sumin d_i\to\muoo,
 \end{equation}
see \eg{} \cite[Theorems 5.4.2 and 5.5.9]{Gut}. 
\end{remark}

We will often need an assumption that is a little stronger than
\refCC{C1d2}.

\begin{condition}
  \label{C2}
As \ntoo, $\nun\to\nuoo$.
(Equivalently, $\E D_n^2\to\E\Doo^2$.)
\end{condition}
This is clearly stronger than \refCC{C1d2}, see \eqref{aa}. Assuming
\refCNN, it is by \eqref{add} equivalent to uniform integrability of
$D_n^2$, cf \refR{RC1}.
In particular, \refCN{C2} holds if $\sup_n\E D_n^ {2+\eps}<\infty$ for
some $\eps>0$.

Let
\begin{align*}
g(x)\=\E x^{\Doo}= \sumk p_k x^k ,
\end{align*}
the probability generating function of the probability distribution
$(p_k)_{k=0}^\infty$. Thus $\muoo=g'(1)$ and $\nuoo=g''(1)$.

We shall use the result by \citet{MR95,MR98} on existence and size of
a giant component in \gndd; we state it in a version from \cite{SJ204}.
For a graph $G$, let 
$v_k(G)$ be the number of vertices of
degree $k$, $k\ge0$.

\begin{theorem}[\citeauthor{MR95}]
  \label{Tgiant}
Suppose that \refCN{C1} holds.
Consider the random  graph \gndd\ and 
let $\cC_1$ and $\cC_2$ be its largest and second largest components.
  \begin{romenumerate}
\item\label{Tgianta}
If\/ $\nuoo-\muoo=\E \Doo(\Doo-2)>0$, then there is a unique $\kk\in(0,1)$
such that $g'(\kk)=\muoo\kk$. With this $\kk$, 
as \ntoo,
\begin{align*}
  \ccc1/n&\pto 1-g(\kk)>0,
\intertext{and}
  v_k(\cC_1)/n&\pto p_k(1-\kk^k), \text{ for every } k\ge0,
 \end{align*}
while $ \ccc2/n\pto 0$. 

\item\label{Tgiantb}
If\/ $\nuoo-\muoo=\E \Doo(\Doo-2)\le0$, then 
$ \ccc1/n\pto 0$.
  \end{romenumerate}

The same results hold for \gnddx.
\end{theorem}

In the usual, somewhat informal, language,
the theorem shows that \gndd{} has a giant component if and only if 
$\nuoo-\muoo=\E \Doo(\Doo-2)>0$. 
We say that $\gndd$ is 
\emph{subcritical} if $\nuoo<\muoo$ ($\E\Doo(\Doo-2)<0$),
\emph{critical} if $\nuoo=\muoo$ ($\E\Doo(\Doo-2)=0$),
\emph{supercritical} if $\nuoo>\muoo$ ($\E\Doo(\Doo-2)>0$).

\begin{remark}
  \label{Rsimple}
  \refCN{C1}\ref{C1l},\ref{C1d2} and \eqref{mn} imply that
\begin{equation*}
\liminf_{\ntoo} \P\bigpar{\gnddx\text{ is a simple graph}}>0;
\end{equation*}
see for instance 
\cite{BenderC},
\cite{BB80},
\cite[Section II.4]{BBbook}, 
\cite{McKay}
and
\cite{McKayWo}
under some extra conditions on $\max d_i$,
and
\cite{SJ195} for the general case.
Since we obtain \gndd{} by conditioning \gnddx{} on being a simple
graph, the results in the present paper for \gndd{}
follow from the results for \gnddx{} by this conditioning.
(We only sometimes state the results for both \gndd\ and \gnddx\ explicitly.)
\end{remark}

\begin{remark}\label{Rp1=0}
  \refCC{C1p1} excludes the case $p_1=0$, when there are some
  pathologies, in particular in the critical case 
(which for $p_1=0$ occurs when $p_0+p_2=1$, \ie, $\doo\in\set{0,2}$ a.s.)
We give some counterexamples for this case in \refS{Sex}, see also
  \cite[Remark 2.7]{SJ204}.

The supercritical case with $p_1=0$ (which occurs as soon as $\P(\doo\ge3)>0$)
is better behaved.
In this case, \refT{Tgiant} holds with $\kk=0$ and thus $\ccc1=n-o(n)$,
see   \cite[Remark 2.7]{SJ204};
hence $\sus(\gndd)=n-o(n)$. We conjecture that
$\susq(\gndd)\pto0$ always in this case, but we have not verified it.
One important example is the random $d$-regular graph \gnd, when all
$d_i=d$ for some fixed $d\ge3$. 
In fact \cite{BB81,Wormald81b}, for $d\ge3$, \whp\ \gnd\ is connected and thus
trivially $\sus(\gnd)=n\to\infty$ and $\susq(\gnd)=0$, in accordance
with \refT{Tsus} (with $\kk=0$).
\end{remark}

\section{Branching processes}\label{Sbranch}

For standard material on branching processes, see \eg\ \cite{AN}.
We review some basic facts that are important for us.
The branching processes that we will use are Galton--Watson processes
where the initial individual has a special offspring distribution.
They are in general defined as follows.

Let $\xio$ and $\xi$ be two given nonnegative integer-valued
random variables (only their
distributions matter). Start the branching process $\bp$ with one
individual in generation 0, and give it a random number $\xio$ of children.
In the sequel, give each individual a number of children that is
distributed as $\xi$, with all these numbers independent.

We let $\qbp$ denote the
total population size of $\bp$, and define  
\begin{equation*}
  \rhok=\rhok(\bp)\=\P(|\bp|=k), \qquad 1\le k\le\infty.
\end{equation*}
In particular, $\rhooo$ is the survival probability of \bp, \ie, the
probability that $\bp$ lives for ever. 

Let $G_0(x)\=\E x^{\xio}$ and $G(x)\=\E x^{\xi}$ be the \pgf{s} of
$\xio$ and $\xi$.
We define 
$\kk$ as the smallest non-negative solution to
\begin{equation}\label{kkk}
G(\kk)=\kk.   
\end{equation}
For a standard Galton--Watson process ($\xio=\xi$), 
it is well-known that this is the extinction probability. 
In general,
by conditioning on $\xio$,
\begin{equation}\label{k37}
  1-\rhooo = \P(\qbp<\infty)=\E\kk^{\xio}=G_0(\kk).
\end{equation}

The susceptibility and modified susceptibility are defined by
\begin{align} \label{susbp0}
 \sus(\bp)
&\=\E\bigpar{|\bp|}
=\sumkoo k\rhok,
\\
  \susqbp
&\=\E\bigpar{|\bp|; {|\bp|<\infty}}
=\sumko k\rhok. \label{susqbp0}
\end{align}
(Note that these are expectations and not random variables.)
Thus, $\susbp=\susqbp$ when the survival probability
$\rhooo=0$ (the subcritical or critical case), 
and
$\susbp=\infty\ge\susqbp$ when $\rhooo>0$ (the supercritical case).

For our random graph $\gndd$ with a given degree sequence satifying
\refCNN, we define the corresponding branching process as the
Galton--Watson branching process  with initial offspring
distribution $\xio\=\Doo$, and general offspring distribution
$\xi=\dx$, where $\dx$ is the shifted size-biased version (or
transform) of $\doo$ defined by
\begin{equation}
  \label{dx}
\P(\dx=k)=\frac{(k+1)\P(\doo=k+1)}{\E\doo},
\qquad k\ge0.
\end{equation}
(We assume \refCNN, so $0<\E \Doo<\infty$ and then \eqref{dx} defines a
probability distribution.) The reason for this definition is the
well-known fact that \dx\ appears as the natural limit distribution
when exploring components locally; the novice can see this in the
proof of \refL{Lnk} below.

Note that
\begin{equation}
  \label{edx}
  \begin{split}
\E\dx
&=
\sumk k\P(\dx=k)
=\sumk \frac{k(k+1)\P(\doo=k+1)}{\E\doo}
\\&
=\frac{\E \doo(\doo-1)}{\E\doo}
=\frac{\nuoo}{\muoo}.
	  \end{split}
\end{equation}
Hence, the standard classification of $\bp$ as subcritical, critical or
supercritical depending on whether the expected number of children
satisfies $\E\xi<1$, $\E\xi=1$ or $\E\xi>1$, becomes the conditions
$\nuoo<\muoo$, $\nuoo=\muoo$ and $\nuoo>\muoo$ we already have seen
for \gndd, and there is a perfect agreement between these types
for \bp\ and for \gndd.
(This indicates that it really is $\nuoo/\muoo$ rather than
$\nuoo-\muoo$ that is the natural parameter for criticality testing
for \gndd; this is well-known, see \eg\ \cite{SJ199} for generalizations.)

Furthermore, if $\gx(x)\=\E x^{\dx}$ is the probability generating
function of $\dx$, then
\begin{equation}
  \label{gx}
G(x)=\gx(x)
=\sumk \frac{(k+1)\P(\doo=k+1)x^k}{\E\doo}
=\frac{g'(x)}{\E\doo}.
\end{equation}
(Note that $\E\doo=g'(1)$, so $\gx(1)=1$ as it should.)
In particular,
\eqref{gx} shows that \eqref{kkk} can be written
\begin{equation}
  {g'(\kk)}={\muoo}\kk.\label{kk}
\end{equation}
Thus, in the supercritical case $\kk$ here is the same as in
\refT{Tgiant}. (In the subcritical and critical cases, $\kk=1$,
which always satifies \eqref{kk}.)
Further, by \eqref{k37}, the asymptotic relative size $1-g(\kk)$ of $\cc1$ in
\refT{Tgiant}(i) equals $\rhooo$ for the corresponding branching
process \bp. (Recall that $G=g$.)

We can easily compute the susceptibility of a Galton--Watson process
by standard calculations.
We consider the general version with $\xio$ and $\xi$, before specializing to
the branching process corresponding to \gndd.
This calculation has been done by \citet{NSW01}, see also
\citetq{Section 2.3}{Durrett}.
(Similar results for a more complicated branching process with types,
but without a special initial offspring distribution, are given in
\cite{SJ232}.) 
\begin{theorem}\label{Tbp}
  For a branching process \bp\ defined as above by $\xio$ and $\xi$, 
  \begin{equation}
	\label{susbp}
\sus(\bp)\=\E(\qbp)
=1+\frac{\E\xio}{(1-\E\xi)_+}.
  \end{equation}
Further, if \/$\xio$ and $\xi$ have the \pgf{s} $G_0$ and $G$, 
and $\kk$ is the smallest nonnegative root of $G(\kk)=\kk$, then
  \begin{equation}
	\label{susqbp}
\susq(\bp)\=\E(\qbp;\qbp<\infty)
=G_0(\kk)+\frac{\kk G_0'(\kk)}{1-G'(\kk)}.
  \end{equation}
Hence, assuming $\E\xio>0$, we have
$\chi(\bp)=\infty$ if and only if \/$\E\xi\ge1$, while
$\susq(\bp)<\infty$ whenever $\E\xi\neq1$.
\end{theorem}

\begin{proof}
As said above, this is proved by \citet{NSW01} (in slightly different
  notation), but for completeness  we give a proof.
  Let $\bp_k$ be the $k$th generation of \bp. Then
  $\E|\bp_k|=\E\xio(\E\xi)^{k-1}$ for $k\ge1$, and thus when $\E\xi\le1$
  (so $\kk=1$ and $\rhooo=0$)
  \begin{equation*}
	\sus(\bp)=1+\sumki\E\xio(\E\xi)^{k-1}
=1+\frac{\E\xio}{1-\E\xi},
  \end{equation*}
while $\sus(\bp)=\infty$ when $\E\xi>1$ and thus $\rhooo>0$.
This shows \eqref{susbp}.

For $\susq$ we use the standard and easily verified fact that
$\hbp\=(\bp\mid\qbp<\infty)$, \ie{} \bp\ conditioned on extinction, is
another branching process with initial offspring distribution $\hxio$
and general offspring distribution $\hxi$ given by
\begin{align*}
  \P(\hxi=k) &= \frac{\kk^k\P(\xi=k)}{G(\kk)} = \kk^{k-1}\P(\xi=k),
\\
  \P(\hxio=k)& = \frac{\kk^k\P(\xio=k)}{G_0(\kk)};
\end{align*}
these have expectations
\begin{align}
  \E\hxi&=G'(\kk),\label{pu1}
\\
\E\hxio&=
\frac{\kk G'_0(\kk)}{G_0(\kk)}. \label{pu0}
\end{align}

If $\E\xi\le1$, then $\kk=1$ and $\hxi\eqd\xi$, $\hxio\eqd\xio$; thus
\eqref{susqbp} reduces to \eqref{susbp}. (Except in the trivial case
$\P(\xi=1)=1$, when $\qbp=1$ or $\infty$ a.s.; then
$\kk=0$ and we interpret \eqref{susqbp} as
$\sus(\hbp)=\P(\qbp<\infty)=G_0(0)$, which is immediate.)

Note that, by the convexity of $G$, if $\E\xi=G'(1)>1$, so $\kk<1$,
then $\E\hxi=G'(\kk)<1$, and $\hbp$ is subcritical.
Thus \eqref{k37}, \eqref{susbp} and \eqref{pu1}--\eqref{pu0} yield
\begin{equation*}
  \begin{split}
	\susq(\bp)&
=\P(\qbp<\infty)\E(\qbp\mid\qbp<\infty)
=\P(\qbp<\infty)\E\qhbp
\\&
=\P(\qbp<\infty)\sus(\hbp)
=G_0(\kk)\Bigpar{1+\frac{\E\hxio}{1-\E\hxi}}
=G_0(\kk)+\frac{\kk G'_0(\kk)}{1-G'(\kk)}.
  \end{split}
\end{equation*}
(The  case $G_0(\kk)=0$, which occurs when
$\P(\xio=0)=\P(\xi=0)=0$, and entails $\kk=0$ and $\qbp=\infty$ a.s.,
is trivial and easily verified separately.)
\end{proof}

We now specialize to the branching process corresponding to \gndd.
\begin{corollary}\label{Cbp}
  Given $\dd=\dn$ such that \refCNN{} holds,  let
  \bp\ be the corresponding branching process. Then
  \begin{align}
	\label{b5a}
\sus(\bp)
&=1+\frac{\muoo}{(1-\nuoo/\muoo)_+}
=1+\frac{\muoo^2}{(\muoo-\nuoo)_+}
\\
\susq(\bp)
&=g(\kk)+\frac{\kk g'(\kk)}{1-g''(\kk)/\muoo}
=g(\kk)+\frac{\kk g'(\kk)g'(1)}{g'(1)-g''(\kk)}
\label{b5b}\\&
=g(\kk)+\frac{\kk g'(\kk)^2}{g'(\kk)-\kk g''(\kk)}
,	\label{b5c}
  \end{align}
with $\susq(\bp)<\infty$ unless $\nuoo=\muoo$.
\end{corollary}

\begin{proof}
  First, 
$\E\xio=\E\Doo=\muoo$ and, by \eqref{edx}, $\E\xi=\E\dx=\nuoo/\muoo$;
  thus \eqref{susbp} yields
  \eqref{b5a}.

Next, $G_0(x)=g(x)$ and, by \eqref{gx}, $G(x)=g'(x)/\muoo$.
Hence, \eqref{susqbp} yields, using \eqref{kk},
  \begin{align*}
\susq(\bp)
&=g(\kk)+\frac{\kk g'(\kk)}{1-g''(\kk)/\muoo}
=g(\kk)+\frac{\kk g'(\kk)^2}{g'(\kk)-\kk g''(\kk)},
  \end{align*}
and the result follows, recalling also $\muoo=g'(1)$.
\end{proof}

The values $\sus(\bp)$ and $\susq(\bp)$ are thus the quantities
called $\susoo$ and $\susqoo$ in 
\refT{Tsus}, and \refT{Tsus} may thus be reformulated as follows.
\begin{theorem}
  \label{Tsus2}
Suppose that Conditions \ref{C1} and \ref{C2} hold, and let \bp\ be
the corresponding branching process.
Then
$\sus(\gndd)\pto\sus(\bp)$
and
$\susq(\gndd)\pto\susq(\bp)$.
In the subcritical case $\nuoo<\muoo$, further 
$\sus(\gndd)\lito\sus(\bp)$
and
$\susq(\gndd)\lito\susq(\bp)$.
\end{theorem}

\section{A lower bound}\label{Slow}
We continue to assume \refCNN, and let $\bp$ be the branching process
corresponding to \gndd\ as in \refS{Sbranch}.
Let $\nk(G)$ denote the number of vertices in components of order
$k$ in a graph $G$. Thus the number of such components is
$\nk(G)/k$.
We can write the definition \eqref{sus} as
\begin{equation}\label{susnk}
  \sus(G)=\frac1{|G|}\sumki \frac{\nk(G)}k k^2=\sumki k\frac{\nk(G)}{|G|}.
\end{equation}

\begin{lemma}
  \label{Lnk}
Suppose that \refCNN\ holds. Then,
for every fixed $k\ge0$, 
\begin{equation}
  \label{nklim}
\nk(\gndd)/n\pto\rhok(\bp)\=\P(\qbp=k).
\end{equation}
The same holds for \gnddx.
\end{lemma}

\begin{proof}
  This is well-known (and see \eg\ \cite{SJ199} for a more general
  situation), but for completeness we sketch the proof. 
By \refR{Rsimple}, it suffices to consider \gnddx.

The expectation $\E \nk(\gnddx)/n$ is the probability that a random
vertex $v_0$ belongs to a component with exactly $k$ vertices. 
We let $\nk'$ be the number of vertices in tree components of order
$k$, and note 
that it is easy to see that the expected number of cycles of length
$\le k$ is $O(1)$, and thus $\E|\nk(\gnddx)-\nk'|=O(1)$; hence it
suffices to consider $\nk'$.

Let $\cC$ be the component containing the random vertex $v_0$.
We explore $\cC$  by breadth-first search, using a
predetermined order of the half-edges at each vertex. In this way, $\cC$
is exhibited as an ordered (or plane) tree, possibly with
some extra edges, and with root $v_0$. Let $T$ be a given tree with
$k$ vertices, and let us compute the probability that $\cC$ equals
$T$ (as an ordered, rooted, unlabelled tree). If $T$ has a root of
degree $d_0$ and $k-1$ other vertices of outdegrees $d_1,\dots,d_k$
(in breadth-first order), then 
there are $n_{d_0}$ choices of $v_0$ and, for $i=1,\dots,k$,
$n_{d_i}-O(1)$ choices of the $i$th vertex. Moreover, for $i\ge1$ we
also have $d_i+1$ choices of half-edge to connect to, out of $2m-O(1)$
remaining half-edges. The probability
is thus, using \eqref{mn} and \eqref{dx},
\begin{equation*}
  \frac{n_{d_0}}n\prod_{i=1}^{k-1}\frac{(d_i+1)n_{d_i+1}}{2m}+o(1)
=  \P(\Doo=d_0)\prod_{i=1}^{k-1}\P(\dx=d_i)+o(1), 
\end{equation*}
which equals, except for $o(1)$, the probability that the family tree
of $\bp$ (considered as an ordered tree) 
equals $T$.
Summing over all trees $T$ of order $k$, we find 
\begin{equation}\label{enk}
  \begin{split}
\E \nk'/n
&=
\P(\cC\text{ is a tree of order $k$)}=\P(\qbp=k)+o(1)=\rhok(\bp)+o(1)
\\&
\to\rhok(\bp).  	
  \end{split}
\raisetag{\baselineskip}
\end{equation}
The same argument, but starting with two independent random vertices,
shows that $\E \nk'(\nk'-k)/n^2\to\rhok(\bp)^2$. Hence,
$\Var(\nk'/n)\to0$, and thus, by \eqref{enk},
$\nk'/n\pto\rhok(\bp)$, which as said above completes the proof.
\end{proof}

\begin{lemma}
  \label{Lner}
Suppose that \refCNN\ holds, and let $a$ be a real number.
\begin{romenumerate}
  \item
If $a<\sus(\bp)$, then $\sus(\gnddx)>a$ \whp
  \item
If further $\nuoo\le\muoo$, then also $\susq(\gnddx)>a$ \whp
\end{romenumerate}
\end{lemma}

\begin{proof}
  In the supercritical case $\nuoo>\muoo$, 
$\sus(\gnddx)\ge\frac1n\ccc1^2=n(\rhooo^2+\op(1)) >a$ \whp\ by \refT{Tgiant}.

Assume thus $\nuoo\le\muoo$, and thus $\rhooo=0$.
By the assumption, \eqref{susbp0} and $\rhooo=0$, then
$a<\sus(\bp)
=\sumko k\rhok$.
Hence there exists $k_0$ such that $\sum_{k=1}^{k_0}k\rho_k >a$.
By \eqref{susnk} and \refL{Lnk}, 
\begin{equation*}
  \begin{split}
\sus(\gnddx)
\=\frac1n\sumki k N_k	
\ge \sum_{k=1}^{k_0}\frac{kN_k}n
\pto
\sum_{k=1}^{k_0}k\rho_k >a,
  \end{split}
\end{equation*}
and thus $\sus(\gnddx)>a$ \whp

Finally, considering the two cases $\ccc1>k_0$ and $\ccc1\le k_0$
separately, we see that
\begin{equation*}
  \begin{split}
\susq(\gnddx)
\ge\frac1n\sum_{k=1}^{k_0}k N_k	-\frac{k_0^2}n
\pto
\sum_{k=1}^{k_0}k\rho_k >a,
  \end{split}
\end{equation*}
and thus also $\susq(\gnddx)>a$ \whp
\end{proof}

\section{An upper bound}\label{Supper}

A path of length $\ell\ge0$ in a multigraph is a sequence
$i_0e_1i_1\dotsm e_\ell i_\ell$ of alternating vertices and
edges that are distinct and such that each $e_j$ has endpoints
$i_{j-1}$ and $i_j$. 

\begin{lemma}\label{LEP}
  Let $P_\ell$ be the number of paths of length $\ell$ in
  $\gnddx$. Then, for every $\ell\ge1$,
  \begin{equation}\label{ep4}
\E P_\ell 
\le \frac{(n\nun)^{\ell-1}}{(n\mun)^{\ell-2}}	
= n\frac{\nun^{\ell-1}}{\mun^{\ell-2}}	.
  \end{equation}
\end{lemma}

\begin{proof}
If $\ell=1$, then \eqref{ep4} says that $\E P_1\le n\mun=\sum_i d_i =2m$;
this is trivially true: $P_1\le 2m$ because a path of length 1 is a
single edge and each edge that is not a
loop yields two paths in opposite direction.

Let $\ell\ge2$ and $\ell\le m$.
A path $i_0\dotsm i_\ell$ contains one half-edge at $i_0$ and at
$i_\ell$, and two half-edges at each of $i_1,\dots,i_\elli$; these may
be chosen arbitrarily, and for each choice, the probability that they
are connected to each other in the right way is
$(2m-1)\qw\dotsm(2m-2\ell+1)\qw$.
Hence,
\begin{equation*}
  \E P_\ell =
\frac{\sum^*_{i_0,\dots, i_\ell} d_{i_0} \cdot
\prod_{j=1}^{\ell-1} d_{i_j}(d_{i_j}-1) \cdot d_{i_\ell}}
{\prod_{j=1}^\ell(2m-2j+1)},
\end{equation*}
where $\sum^*$ denotes the sum over \emph{distinct} indices.

For each choice of distinct $i_0,\dots,i_{\ell-1}$ with $d_{i_0}\ge1$
and $d_{i_j}\ge2$, $1\le j\le \ell-1$, the sum $\sum d_{i_\ell}$ over
$i_\ell\notin\set{i_0,\dots,i_\elli}$ equals 
\begin{equation*}
  \sumin d_i-d_{i_0} -   \sum_{j=1}^{\ell-1}d_{i_j}
\le 2m-1-2(\ell-1)
=2m-2\ell+1.
\end{equation*}
Hence,
\begin{equation*}
  \E P_\ell 
\le
\frac{\sum^*_{i_0,\dots, i_{\ell-1}} d_{i_0}
\prod_{j=1}^{\ell-1} d_{i_j}(d_{i_j}-1) }
{\prod_{j=1}^{\ell-1}(2m-2j+1)}.
\end{equation*}

Similarly, for each choice of distinct $i_1,\dots,i_{\ell-1}$ with 
$d_{i_j}\ge2$, $1\le j\le \ell-1$, the sum $\sum d_{i_0}$ over
$i_0\notin\set{i_1,\dots,i_{\ell-1}}$ equals 
\begin{equation*}
  \sumin d_i- \sum_{j=1}^{\ell-1}d_{i_j}
\le 2m-2(\ell-1)
<2m-2\ell+3.
\end{equation*}
Hence,
\begin{equation}\label{ep3}
  \E P_\ell 
\le
\frac{\sum^*_{i_1,\dots, i_{\ell-1}} 
\prod_{j=1}^{\ell-1} d_{i_j}(d_{i_j}-1) }
{\prod_{j=1}^{\ell-2}(2m-2j+1)}
\le
2^{-(\ell-2)}
\frac{\sum^*_{i_1,\dots, i_{\ell-1}} 
\prod_{j=1}^{\ell-1} d_{i_j}(d_{i_j}-1) }
{\prod_{j=1}^{\ell-2}(m-j)}
\end{equation}

Let $a_i\=d_i(d_i-1)$ and let $R$ be the set of indices $i$ such that
$a_i>0$, \ie, $d_i\ge2$. Let $r\=|R|$.
By an inequality of Maclaurin \cite[Theorem 52]{HLP},
for $2\le\ell\le r+1$,
\begin{equation*}
  \lrpar{\frac{(r-\ell+1)!}{r!}\sumx_{i_1,\dots,i_{\ell-1}\in R}
\prod_{j=1}^\elli a_{i_j} }^{1/(\ell-1)}
\le\frac1r\sum_{i\in R} a_i
\le\frac1r\sumin a_i
=\frac{n\nun}{r}.
\end{equation*}
Hence,
\begin{equation*}
  \begin{split}
\sumx_{i_1,\dots,i_{\ell-1}}\prod_{j=1}^\elli d_{i_j}(d_{i_j}-1)
&=
\sumx_{i_1,\dots,i_{\ell-1}\in R}\prod_{j=1}^\elli a_{i_j}
\le\Bigparfrac{n\nun}{r}^\elli \frac{r!}{(r-\ell+1)!}
\\&
=\xpar{n\nun}^\elli\prod_{j=0}^\ellii\Bigpar{1-\frac jr}.	
  \end{split}
\end{equation*}
Further, $2m=\sumin d_i \ge 2r$, so $r\le m$.
Consequently, \eqref{ep3} yields
\begin{equation*}
  \E P_\ell 
\le
(2m)^{-(\ell-2)}
\frac
{\prod_{j=1}^{\ell-2}(1-j/r)}
{\prod_{j=1}^{\ell-2}(1-j/m)}
\cdot(n\nun)^\elli
\le
\frac{(n\nun)^\elli}{(2m)^\ellii}
=
\frac{(n\nun)^\elli}{(n\mun)^\ellii},
\end{equation*}
which proves the result when $2\le \ell\le m$ and $\ell\le r+1$.
Since trivially $P_\ell=0$ if $\ell>m$ or $\elli>r$, this completes
the proof.
\end{proof}

\begin{lemma}
  \label{Lupp}
For any $\dd$,
  \begin{equation}\label{lupp}
	\E\bigpar{\sus(\gnddx)}\le 1+\frac{\mun^ 2}{(\mun-\nun)_+}.
  \end{equation}
\end{lemma}

\begin{proof}
  By \cite[Lemma 4.6]{SJ232} (which trivially extends to multigraphs),
  \begin{equation*}
	{\sus(\gnddx)}
\le \frac1n\sum_{\ell=0}^\infty P_\ell(\gnddx)
=1+ \frac1n\sum_{\ell=1}^\infty P_\ell(\gnddx).
  \end{equation*}
Hence, using 
\refL{LEP} for $\ell\ge1$,
  \begin{equation*}
	\E\bigpar{\sus(\gnddx)}\le 
1+\sum_{\ell=1}^\infty\frac{\E P_\ell(\gnddx)}n
\le
1+\sum_{\ell=1}^\infty\frac{\nun^\elli}{\mun^\ellii}.
  \end{equation*}
If $\mun>\nun$, the geometric series sums to $\mun/(1-\nun/\mun)=\mun^
2/(\mun-\nun)$, and the result follows.
The case $\nun\ge\mun>0$ is trivial, since the right hand side of
\eqref{lupp} is $\infty$, and the
case $\mun=0$ is trivial too, since then there are no edges at all and
thus $\sus(\gnddx)=1$. 
\end{proof}

\section{Duality}\label{Sdual}
Let, as before, $\cc1$ be the largest component of \gndd; if there is
a tie we for definiteness choose the component with maximal size that
contains the vertex with largest label.
Consider the complement of $\cc1$; we denote this random graph by
$\hG\ndd\=\gndd\setminus\cc1$. This graph has thus the random vertex
set $[n]\setminus\cc1\co$, where $[n]\=\set{1,\dots,n}$ and we in this
section denote the vertex set of a graph $\cC$ by $\cC\co$.

We have defined \gndd\ with the vertex set $[n]$. 
Of course, the definition generalizes to an arbitrary finite vertex
set $A$ and a degree sequence $\dd=(d_i)_{i\in A}$; we denote this
random graph by $G(A,\dd)$.
Is $A$ is a subset of $[n]$, and $\dd=\dn$, let $\dd\restr A$ be the
sequence $(d_i)_{i\in A}$.

We  construct, given $n$ and $\dd=\dn$, a random graph
$\tG\ndd$ by first constructing $\gndd$ and finding its largest
component $\cc1$; we then, given $\cc1$, 
let $A\=[n]\setminus\cc1$ and 
construct a new random graph $G(A,\dd\restr{A})$
and take that as our random graph $\tG\ndd$.
Hence $\tG\ndd$ is a random graph where both the vertex set and the
edge set are random, but conditioned on the vertex set, it is a
uniform random
graph with given vertex degrees.
Our version of the duality theorem is that $\hG\ndd$ and $\tG\ndd$
are equal \whp, with a suitable coupling. This is a precise version of
saying that $\hG\ndd$ conditioned on its vertex set almost is a
uniform random graph with given vertex degrees.

\begin{theorem}
  \label{Tdual}
Suppose that \refCNN\ holds, and that $\nuoo>\muoo$.
With the notations above, it is possible to couple $\hG\ndd$ and
$\tG\ndd$ such that they coincide w.h.p.
Furthermore, we may assume, by another coupling, that $\tG\ndd$
conditioned on its order and degree sequence satisfies \refCNN,
with $p_k$ replaced by
$\hp_k\= p_k\kk^k/g(\kk)$.
Let $\hD$ be a random variable with this distribution:
\begin{equation}
  \P(\hD=k)=\hp_k\=\frac{ p_k\kk^k}{g(\kk)},
\qquad k\ge0.
\end{equation}
Then $\hD$ has \pgf
\begin{equation}\label{hg}
  \hg(x)\=\E x^{\hD}=\frac{g(\kk x)}{g(\kk)},
\end{equation}
and
\begin{align}
  \hmuoo&\=\E\hD=\hg'(1)=\frac{\kk g'(\kk)}{g(\kk)}
=\frac{\kk^2 \muoo}{g(\kk)},
\label{muq}\\
\hnuoo&\= \E\hD(\hD-1)=\hg''(1)=\frac{\kk^2 g''(\kk)}{g(\kk)}.
\label{nuq}
\end{align}
Moreover, $\hnuoo<\hmuoo$, so $\tG\ndd$ is subcritical.
\end{theorem}

Before giving the proof, we give a simple and well-known result on
conditioning. Recall that the \emph{total variation distance} between
two random variables $X$ and $Y$ (taking values in any common 
space) is 
$$\dtv(X,Y)\=\sup_A\bigpar{\P(X\in A)-\P(Y\in A)},$$ 
taking
the supremum over all measurable sets $A$. Recall further that the
existence of a coupling with $\hG\ndd=\tG\ndd$ \whp\ is equivalent to
$\dtv(\hG\ndd,\tG\ndd)\to0$.

\begin{lemma}
  \label{Ldtv}
If $X$ is any random variable (with values in any space) and $\cE$ is
any event with $\P(\cE)>0$, then
$\dtv\bigpar{(X\mid \cE),X}\le 1-\P(E)$.
\end{lemma}

\begin{proof}
  For any event $\cA$ of the type \set{X\in A},
  \begin{equation*}
	\P(\cA\mid\cE)-\P(\cA)
=\frac{\P(\cA\cap\cE)}{\P(\cE)}-\P(\cA)
\le
\frac{\P(\cA)\bmin\P(\cE)}{\P(\cE)}-\P(\cA).
  \end{equation*}
The \rhs\ is a function of $\P(\cA)$ that is maximal for $\P(\cA)=\P(\cE)$,
when it equals $1-\P(\cE)$.
\end{proof}

\begin{proof}[Proof of \refT{Tdual}]
  Define a total order $\prec$ on the subsets of $\xn$ by defining
  $A\prec B$ if $|A|<|B|$ or $|A|=|B|$ and $\max A<\max B$.
Thus $\cC_1$ is by definition the component of \gndd{} whose vertex
  set is maximal in this order.

Let $A\subseteq \vn$.
Conditioned on $\cc1\co=A$, the complement
$\gndd\setminus\cc1=\gndd\restr{\vn\setminus A}$ is a random graph on
the vertex set $\vna$ with a given degree sequence $(d_i)_{i\in\vna}$.
Moreover, it may be any such graph except that it must not contain a
component $\cC$ with $\cC\co\succ A$; furthermore, all permitted
graphs have the same probability.
Thus, conditioned on $\cc1\co=A$, 
\begin{equation*}
  \hG\ndd=\gndd\restr{\vna}
\eqd\bigpar{G(\vna,\dd\restr{\vna})\bigm|\heac},
\end{equation*}
where $\heac$ is the complement of the event $\hea$ that 
$G(\vna,\dd\restr{\vna})$ contains a component $\cC$ with
$\cC\co\succ A$.

On the other hand, by definition, conditioned on $\cc1\co=A$
we have $\tG(n,\dd)=G(\vna,\dd\restr{\vna})$.
Hence, by \refL{Ldtv}, the total variation distance between
$\hG\ndd$ and $\tG\ndd$, both conditioned on $\cc1\co=A$, is
{\multlinegap=0pt
	\begin{multline*}
\dtv\Bigpar{\bigpar{\hG\ndd\mid\cc1\co=A},
\bigpar{\tG\ndd\mid\cc1\co=A}}
\\=
\dtv\Bigpar{
\bigpar{G(\vna,\dd\restr{\vna})\bigm|\heac},
G(\vna,\dd\restr{\vna})}	
\le1-\P(\heac)=\P(\hea).	  
	\end{multline*}}
Taking the expectation over $\cc1\co$ we find
\begin{equation}\label{hp0}
  \begin{split}
\dtv\bigpar{\hG\ndd,\tG\ndd}
&\le \E \dtv\Bigpar{\bigpar{\hG\ndd\mid\cc1\co},
\bigpar{\tG\ndd\mid\cc1\co}}
\\&
\le\E\P(\hex{\cc1\co})
=\sum_{A\subseteq[n]}\P(\hea)\P(\cc1\co=A).
  \end{split}
\end{equation}
We split this sum into two parts.
Let $r\=\rhooo/2$; thus $r>0$ and $\P(\ccc1\ge rn)\to1$.
Further, let $\cea$ be the event that $A=\cC\co$ for some component
$\cC$ of \gndd, and note that $\cc1\co=A$ implies $\cea$.
Thus
\begin{equation}\label{hp2}
  \begin{split}
\sum_{A\subseteq[n]}\P(\hea)\P(\cc1\co=A)
\le\P(\ccc1<rn) + \sum_{|A|\ge rn}\P(\hea)\P(\cea).
  \end{split}
\end{equation}
Conditioned on $\cea$, the complement $\gndd\restr{\vna}$ of $A$ has
the same distribution as $G(\vna,\dd_{\vna})$, and thus
\begin{equation*}
  \begin{split}
\P(\hea)
&=\P\bigpar
 {\text{\gndd\ contains a component $\cC$ with $\cC\co\succ A$}\bigm|\cea}
\\&
=\P\Bigpar{\bigcup_{B\succ A}\ceb\Bigm|\cea}.
  \end{split}
\end{equation*}
Hence, 
\begin{equation}\label{hp1}
  \begin{split}
\sum_{|A|\ge rn}\P(\hea)\P(\cea)
=
\sum_{|A|\ge rn}\P\bigpar{\bigcup_{B\succ A}\ceb\cap\cea}
=
\E\sum_{|A|\ge rn}\etta\bigsqpar{\cea\cap\bigcup_{B\succ A}\ceb}.
  \end{split}
\end{equation}
Let $N$ be the number of components of size $\ge rn$ in \gndd. If we
order these components as $A_1\prec\dots \prec A_N$, then the
indicator in \eqref{hp1} is 1 exactly when $A$ is one of
$A_1,\dots,A_{N-1}$, so the sum is $(N-1)_+$. Further, since components
are disjoint, $N\le n/(rn)=r\qw$, and thus $(N-1)_+\le
r\qw\ett{N\ge2}$. However, $N\ge2$ if and only if the \emph{second}
largest component $\cc2$ is larger than $rn$.
Consequently, by \eqref{hp0}, \eqref{hp2}, \eqref{hp1},
\begin{equation*}
  \begin{split}
\dtv\bigpar{\hG\ndd,\tG\ndd}
&\le\P(\ccc1<rn) 
+r\qw\P(N\ge2)
\\&
=\P(\ccc1<rn) 
+r\qw\P(\ccc2\ge rn),
  \end{split}
\end{equation*}
where, by \refT{Tgiant}, both terms on the \rhs\ tend to 0.
This shows the existence of a copuling with $\hG\ndd=\tG\ndd$ \whp

By the Skorohod coupling theorem \cite[Theorem 4.30]{Kallenberg}, 
we may assume that the random graphs for different $n$ are coupled
such that the limits in \refT{Tgiant}\ref{Tgianta} hold a.s. 
Let $\tilde\dd\=\dd\restr{\xn\setminus \cc1\co}$ be the degree
sequence used to define $\tGx\ndd$, let
$\tn\=|\tGx\ndd|$ be its length and let
$\tn_k$ be the number of elements $d_i=k$ in it.
Then 
$\tn=n-\ccc1$ and $\tn_k=n_k-v_k(\cc1)$, and thus, by \refT{Tgiant}
with the assumed coupling,
$\tn/n\asto g(\kk)$ and $\tn_k/n\asto p_k\kk^k$, $k\ge0$,
and thus $\tn_k/\tn\asto \hp_k\= p_k\kk^k/g(\kk)$. Consequently, 
conditioned on the order and degree sequence of $\tGx\ndd$, \refCNN\
then holds \as, with $p_k$ replaced by
$\hp_k$.

The formulas \eqref{hg}--\eqref{nuq} are straightforward.

Finally, since \gndd{} is supercritical, $p_k>0$ for some $k\ge3$, and
thus $g''(x)$ is strictly increasing. Hence, recalling 
$g'(1)=\muoo$ and $g'(\kk)=\kk\muoo$,
\begin{equation*}
  (1-\kk)\muoo = g'(1)-g'(\kk)=\int_\kk^1g''(x)\ddd x
>(1-\kk)g''(\kk).
\end{equation*}
Consequently, $\muoo>g''(\kk)$ and \eqref{muq}--\eqref{nuq} show that
$\muqoo>\nuqoo$.
\end{proof}

\section{Proof of main theorems}\label{Spf}
\begin{proof}[Proof of Theorems \ref{Tsus} and \ref{Tsus2}]
By \refR{Rsimple}, it suffices to prove Theorems \ref{Tsus} and
\ref{Tsus2} for \gnddx.

Consider first the subcritical case $\nuoo<\muoo$. 
By \refC{Cbp}, $\susoo=\sus(\bp)$ and $\susqoo=\susq(\bp)$.
By \refL{Lner}(i), if
$a<\susoo\=\sus(\bp)$, then $\sus(\gnddx)>a$ \whp, while \refL{Lupp},
\eqref{mn} and \refCN{C2}
show that 
\begin{equation*}
\E\sus(\gnddx) \le 1+\frac{\mun^2}{(\mun-\nun)_+}  
\to 1+\frac{\muoo^2}{\muoo-\nuoo}  =\susoo.
\end{equation*}
These upper and lower bounds imply, see \citetq{Lemma 4.2}{SJ232},
that $\sus(\gnddx)\lito\susoo$. 
The same argument holds for $\susq(\gnddx)$, by Lemmas \ref{Lner}(ii) and
\ref{Lupp} together with $\susq\le\sus$.

In the critical case $\nuoo=\muoo$, \refL{Lner} yields
$\sus(\gnddx)\ge\susq(\gnddx)>a$ \whp{} for any finite $a$, and thus 
$\sus(\gnddx),\,\susq(\gnddx)\pto\infty$.

In the supercritical case, \refL{Lner} shows $\sus(\gnddx)\pto\infty$.
For $\susq$ we consider $\gndd$ and note that
\begin{equation}
  \label{p1}
\susq(\gndd)\=\frac1n\sum_{\cC_i\neq\cc1}\ccc i^2
=\frac{n-\ccc1}n\sus\bigpar{\hG\ndd}.
\end{equation}
Here $(n-\ccc1)/n\pto g(\kk)$ by \refT{Tgiant}, and by \refT{Tdual}, we
may couple $\hG\ndd$ and $\tG\ndd$ such that \whp\ 
they coincide and thus
\begin{equation}
  \label{p3}
\sus\bigpar{\hG\ndd}=\sus\bigpar{\tG\ndd}
\qquad\text{\whp}
\end{equation}

We may by \refT{Tdual} assume that $(n-\ccc1)/n\to g(\kk)$ a.s.\  and that
\refCNN\ holds for $\tG\ndd$, 
conditioned on its order and degree sequence, 
with $p_k$ replaced by
$\hp_k$.
Further, for any constant $A$,
\begin{equation*}
  \frac1{|\tG\ndd|}\sum_{i\in V(\tG\ndd)} d_i^2 \ett{d_i\ge A}
\le \frac{1}{n(g(\kk)+o(1))}\sum_i d_i^ 2\ett{d_i>A},
\end{equation*}
so the uniform integrability of $D_n$ implies that also \refCN{C2}
holds for the random graph $\tG\ndd$ 
conditioned on its order and degree sequence.

We can thus apply the already proven result for $\sus$ to $\tG\ndd$, 
conditioned on its order and degree sequence, and conclude that
\begin{equation}
  \label{p2}
\sus\bigpar{\tG\ndd}
\pto
1+\frac{\hmuoo}{1-\hnuoo/\hmuoo}.
\end{equation}
By \eqref{p1},  \eqref{p3}, \eqref{p2}, 
\eqref{muq}, \eqref{nuq} and \eqref{b5b}, we thus obtain
\begin{equation*}
\susq\bigpar{\gndd}
\pto
g(\kk)\Bigpar{1+\frac{\hmuoo}{1-\hnuoo/\hmuoo}}
=g(\kk)+\frac{\kk g'(\kk)}{1-g''(\kk)/\muoo}
=\susq(\bp). \qedhere
\end{equation*}
\end{proof}

\section{Approximation}\label{Sapp}
We have assumed \refCN{C1}, including convergence of the degree
distribution $D_n$. This is convenient, but it is also interesting to
regard the result as an approximation for finite $n$, without assuming
convergence of $D_n$. For simplicity we consider only $\sus$, leaving
the similar but notationally more complicated result for $\susq$ to
the reader.

In order to treat convergence to $\infty$, we let $\gdd$ be a metric
on the compact space $[1,\infty]$, for example $\gdd(x,y)\=\abs{x\qw-y\qw}$.

\begin{theorem}
  \label{Tsus3}
Suppose that $\dd=\dn$ are given for $n\ge1$ such that the random variables
$D_n$ are uniformly square integrable and that
$\liminf_\ntoo\P(D_n=1)>0$.
Then
\begin{equation}\label{tsus3a}
  \gdd\Bigpar{\sus(\gndd),\,1+\frac{\mun^2}{(\mun-\nun)_+}}
\pto0.
\end{equation}
If further $\mun-\nun\ge c>0$, for some fixed $c$, 
then also
\begin{equation}\label{tsus3b}
  \E\Bigabs{\sus(\gndd)-\Bigpar{1+\frac{\mun^2}{\mun-\nun}}}
\to0,
\end{equation}
and thus
$\sus(\gndd)=1+\xfrac{\mun^2}{(\mun-\nun)}+\op(1)$.
\end{theorem}

\begin{proof}
  The uniform square integrability implies that $\sup_n\E
  D_n^2<\infty$; hence the variables $D_n$ are tight, and we may by
  considering a subsequence assume that $D_n\dto \doo$ for some random
  variable $\doo$ on $\bbZ_{\ge0}$. However, this is exactly
  \refCC{C1p}; furthermore $\P(\doo=1)>0$, and the uniform square
  integrability of $D_n$ implies that Conditions \refand{C1}{C2} hold
(along the subsequence).
In particular, $\mun\to\muoo$ and $\nun\to\nuoo$.

If $\nuoo<\muoo$, then, by \refT{Tsus} applied to the subsequence,
$\sus(\gndd)\lito 1+\muoo^2/(\muoo-\nuoo)$. Further, 
$\mun^2/(\mun-\nun)_+\to\muoo^2/(\muoo-\nuoo)$, and both claims follow.

If $\nuoo\ge\muoo$, we have $\sus(\gndd)\pto\infty$ by \refT{Tsus};
further, $\mun\to\muoo>0$ and $(\mun-\nun)_+\to0$,
and thus
$\mun^2/(\mun-\nun)_+\to\infty$, and \eqref{tsus3a} follows in this
case too.

Hence, there is always a subsequence along which the results
hold. Since we may start by taking an arbitrary subsequence, the
results hold generally by the standard subsubsequence principle, see \eg{}
\cite[p.~12]{JLR}. 
\end{proof}

\section{Approaching criticality}\label{Scrit}

In order to study the critial behaviour more closely, we consider a
family of random graphs parametrized by a parameter $\gl$ besides
$n$. We consider asymptotics as \ntoo, and investigate how the limits
depend on $\gl$. More precisely, we consider for simplicity the
following case.

Let $h(x)$ be the probability generating function of a 
non-negative integer-valued random
variable $D$ such that $\infty>h''(1)=\E D(D-1)>h'(1)=\E D$; this thus
corresponds to a 
supercritical $\gndd$. Now, for a fixed $\gl\in(0,1]$, 
add an even number ${n(1-\gl)/\gl}+O(1)$ of vertices of degree 1; this
gives a random graph $\gndl$ with a corresponding
asymptotic degree distribution $\doo$ that has the \pgf
\begin{equation}\label{ggl}
\ggl(x)=(1-\gl)x+\gl h(x).  
\end{equation}
In particular (we will often omit the argument $\gl$ from the notation),
\begin{align}
\muoo=  \muoo(\gl)&=\ggl'(1)=(1-\gl)+\gl h'(1),\label{mul}
\\
  \nuoo=\nuoo(\gl)&=\ggl''(1)=\gl h''(1).\label{nul}
\end{align}
Consequently, the random graph \gndl\ is critical when $1-\gl+\gl
h'(1)=\gl h''(1)$, 
\ie, when $\gl=\glc$ given by
\begin{equation}\label{glc}
  \glc\=\frac1{1- h'(1)+h''(1)}\in(0,1),
\end{equation}
while it is subcritical for $\gl<\glc$ and supercritical for
$\gl>\glc$.

We consider the limits $\susoo(\gl)$ and $\susqoo(\gl)$ 
of $\sus(\gndl)$ and $\susq(\gndl)$ as \ntoo, and
investigate how they depend on $\gl$ as $\gl\to\glc$. (We thus let
first \ntoo{} and then $\gl\to\glc$. A related problem, not considered
here, is to let $\gl\to\glc$ and $\ntoo$ 
simultaneously.) 

For the subcritical case $\gl<\glc$ we have, by \refT{Tsus} and
\eqref{mul}--\eqref{glc},
\begin{equation*}
  \begin{split}
\susoo
=1+\frac{\muoo^2}{\muoo-\nuoo}	
=1+\frac{\muoo(\gl)^2}{1-\gl+\gl h'(1)-\gl h''(1)}	
=1+\frac{\muoo(\gl)^2}{1-\gl/\glc}	.
  \end{split}
\end{equation*}
For $\gl\upto\glc$, it follows that,
with $\muc\=\muoo(\glc)=\nuoo(\glc)>0$,
\begin{equation}\label{susc}
  \begin{split}
\susoo
\sim\frac{\glc\muoo(\glc)^2}{\glc-\gl}
=\frac{\glc\muc^2}{\glc-\gl}	.
  \end{split}
\end{equation}

In the supercritical case $\gl>\glc$, the parameter $\kk=\kk(\gl)$ is given by
$\ggl'(\kk)=\kk\ggl'(1)$, or, by \eqref{ggl},
\begin{equation}\label{sw1}
  1-\gl+\gl h'(\kk)=\kk(1-\gl+\gl h'(1));
\end{equation}
equivalently,
\begin{equation}\label{sw2}
  1-\kk=\gl(1-h'(\kk)-\kk+\kk h'(1)).
\end{equation}
We use $\eps\=1-\kk$ as parameter, and have thus $\kk=1-\eps$ and, by
\eqref{sw2} and \eqref{glc}, 
\begin{gather}\notag
  \frac1\gl
=
\frac{\eps+h'(1)-\eps h'(1)-h'(1-\eps)}{\eps}
=
1-h'(1)+\frac{h'(1)-h'(1-\eps)}{\eps},
\\
\frac1{\glc}- \frac1\gl
=
h''(1)-\frac{h'(1)-h'(1-\eps)}{\eps}.
\label{allan}
\end{gather}
As $\gl\downto\glc$, $\kk\upto1$ and thus $\eps=1-\kk\downto0$.
Moreover, $\ggl'(\kk)-\kk \ggl''(\kk)\to\gglc'(1)-\gglc''(1)=0$, and
thus by \refT{Tsus}(iii) and $\ggl'(\kk)=\kk\ggl'(1)$,
\begin{equation}\label{magnus}
\susqoo
\sim\frac{\kk(\ggl'(\kk))^ 2}{\ggl'(\kk)-\kk \ggl''(\kk)}
=\frac{(\ggl'(\kk))^ 2}{\ggl'(1)- \ggl''(\kk)}
\sim\frac{\muc^ 2}{\ggl'(1)- \ggl''(\kk)},
\end{equation}
where by \eqref{ggl}
\begin{equation}\label{ika}
  \begin{split}
\ggl'(1)- \ggl''(\kk)
&=
1-\gl+\gl h'(1)-\gl h''(\kk)
=\gl\Bigpar{\frac1\gl - \frac1{\glc}+h''(1)-h''(\kk)}.
  \end{split}
\end{equation}
Let $r(\eps)$ be the remainder term in the Taylor expansion
\begin{equation*}
  h(1-\eps)=1-\eps h'(1)+\frac{\eps^2}2 h''(1)-r(\eps),
\end{equation*}
and note that $r(0)=r'(0)=r''(0)=0$, while $r'''(\eps)=h'''(1-\eps)$.
Then \eqref{allan} and \eqref{ika} can be written
\begin{align}
  \frac1{\glc}-\frac1{\gl}& =\frac{r'(\eps)}{\eps}, \label{ku}
\\
\ggl'(1)- \ggl''(\kk)&\sim \gl\Bigpar{-\frac{r'(\eps)}{\eps}+r''(\eps)}.
\end{align}
These yield, using \eqref{magnus},
\begin{align*}
\gl-\glc&\sim \glc^2\frac{r'(\eps)}{\eps},
\\
\susqoo&\sim\frac{\muc^2}{\ggl'(1)- \ggl''(\kk)}
\sim \glc\qw\muc^2
\Bigpar{-\frac{r'(\eps)}{\eps}+r''(\eps)}\qw
\end{align*}
and thus, as $\gl\downto\glc$,
\begin{equation}\label{sov}
(\gl-\glc)\susqoo
\sim 
\glc\muc^2
 \frac{r'(\eps)}
{\eps r''(\eps)-r'(\eps)}.
\end{equation}

Since $r''(\eps)=\intoe r'''(t)\ddd t$ and
$r'(\eps)=\intoe(\eps-t)r'''(t)\ddd t$, \eqref{sov} can be written
\begin{equation}\label{sof}
\susqoo
\sim 
\frac{\glc\muc^2}{\gl-\glc}\cdot
 \frac{\intoe(\eps-t)r'''(t)\ddd t}
{\intoe t r'''(t)\ddd t}.
\end{equation}

If $\E D^3<\infty$, then with $b\=h'''(1)=\E D(D-1)(D-2)>0$ (because
$h$ is supercritical), as $\eps\to0$,
$r'''(\eps)\sim b$, $r''(\eps)\sim b\eps$
and $r'(\eps)\sim b \eps^2/2$; thus \eqref{sov} or \eqref{sof}
yields, as $\gl\downto\glc$
\begin{equation}\label{supc}
  \susqoo\sim\frac{\glc\muc^2}{\gl-\glc}.
\end{equation}
Combining \eqref{supc} and \eqref{susc}, we obtain as $\gl\to\glc$
from any side 
the symmetric asymptotic, assuming $\E D^3<\infty$,
\begin{equation}\label{symm}
  \susqoo\sim\frac{\glc\muc^2}{\abs{\gl-\glc}}.
\end{equation}
(Hence, in terminology from percolation theory and mathematical
physics, the critical exponent $\gamma$ equals 1.) 
Such symmetry between the subcritical and supercritical sides has been
observed in many different models, but there are also exceptions: for
example, for the CHKNS model, $\susq$ has finite limits as the parameter $\gl$
increases or decreases to the critical value, but the limits are
different and the derivative is finite on the supercritical side but
not on the subcritical side where there is a square-root singularity,
see \citet{CHKNS} and \cite[Section 6.3]{SJ232}. 
We shall see that also for $\gndd$ it is possible to have asymmetric
asymptotics when $\E D^3=\infty$. We consider some examples, with
less and less integrability of $D$ beyond $\E D^2<\infty$, which we
always assume. We let for convenience $c_1,c_2,\dots$ denote some
positive constants (depending on $D$) whose values we do not want to
write explicitly. 

\begin{example}
  \label{Easymm1}
Let $\P(D=k)\sim k^{-3-\ga}$  as $k\to\infty$, 
for some
$\ga\in(0,1)$. Then 
$r'''(\eps)\sim\Gamma(1-\ga)\eps^{\ga-1}$ as \epsto, and 
thus 
\eqref{sof} yields, as $\gl\downto\glc$,
\begin{equation*}
\susqoo
\sim 
\frac{\glc\muc^2}{\gl-\glc}\cdot
 \frac{\intoe(\eps-t)t^{\ga-1}\ddd t}
{\intoe t^{\ga}\ddd t}
=
\frac{\glc\muc^2}{\gl-\glc}\cdot
 \frac{\intoi(1-t)t^{\ga-1}\ddd t}
{\intoi t^{\ga}\ddd t}
=
\frac{\ga\qw\glc\muc^2}{\gl-\glc},
\end{equation*}
with the same exponent as on the subcritical side but a different constant.
\end{example}

\begin{example}
  \label{Easymm2}
Let $\P(D=k)\sim k^{-3}(\log k)^{-\ga}$  as $k\to\infty$, 
for some $\ga>1$. 
Then $r'''(\eps)\sim 1/(\eps|\loge|^\ga)$ as \epsto,  
and thus 
\begin{equation*}
\eps  r''(\eps)-{r'(\eps)} =\int_0^\eps t r'''(t)\ddd t
\sim\frac{\eps}{|\loge|^\ga},
\end{equation*}
while
$r'(\eps)\sim (\ga-1)\qw\eps/|\loge|^{\ga-1}$.
Hence, by \eqref{sov}, 
\begin{equation*}
\susqoo
\sim 
\frac{\CC}{\gl-\glc}|\loge|.
\end{equation*}
Further, \eqref{ku} yields 
\begin{equation*}
  \gl-\glc \sim \glc\qww \frac{r'(\eps)}{\eps} \sim \frac{\CC}{|\loge|^{\ga-1}}
\end{equation*}
and thus $|\loge|\sim \CC(\gl-\glc)^{-1/(\ga-1)}$.
Consequently, as $\gl\downto\glc$,
\begin{equation*}
\susqoo
\sim 
{\CC}{(\gl-\glc)^{-\ga/(\ga-1)}},
\end{equation*}
with a critical exponent $\ga/(\ga-1)>1$, in contrast to the
subcritical case \eqref{susc}.
\end{example}

\begin{example}
  \label{Easymm3}
Let $\P(D=k)\sim k^{-3}(\log k)\qw (\log\log k)\qww$  as $k\to\infty$.
Then, 
$r'''(\eps)\sim\eps\qw\aloge\qw(\log\aloge)\qww$, and thus
$r''(\eps)\sim1/(\log\aloge)$ and $r'(\eps)\sim\eps/(\log\aloge)$, while
\begin{equation*}
 \eps r''(\eps)-r'(\eps) = \int_0^\eps t r'''(t)\ddd t
  \sim\frac{\eps}{\aloge\cdot(\log\aloge)^2}.
\end{equation*}
Hence, \eqref{sov} yields, as $\gl\downto\glc$,
\begin{equation}\label{qk}
  \susqoo\sim \frac{\CC}{\gl-\glc} \aloge\cdot\log\aloge.
\end{equation}

Furthermore, \eqref{ku} yields
\begin{equation*}
  \gl-\glc \sim \glc^2 \frac{r'(\eps)}{\eps}
\sim \frac{\glc^2}{\log\aloge}.
\end{equation*}
and thus $\aloge = \exp\bigpar{(\CC+o(1))/(\gl-\glc)}$.
Consequently \eqref{qk} yields
\begin{equation*}
  \susqoo= e^{(\CCx+o(1))/(\gl-\glc)},
\end{equation*}
with a more rapid growth than any power of $(\gl-\glc)\qw$. (The
critical exponent is $\infty$.)
\end{example}

It seems that in this way we can find examples where $\susqoo$ grows
arbitrarily fast as $\gl\downto\glc$.

\section{Some counterexamples}\label{Sex}

We give some examples where \refCNB\ are not satisfied, in order to show
that the results in general do not hold without these conditions.

\begin{example}\label{Estar}
  Let $d_i=1$ for $i\ge2$, and $d_1\sim a\sqrt n$ for some $a>0$. Thus
  $\gndd$ has $n-1$ vertices of degree 1 and a single vertex of higher
  degree $d_1$. Consequently, the components are a single star with
  $d_1+1$ vertices and $(n_1-1-d_1)/2$ isolated edges. Hence
  (deterministically), 
  \begin{equation}
	\sus\bigpar{\gndd}
=\frac1n\bigpar{(d_1+1)^2+2(n-1-d_1)} \to a^2+2.
\label{estar}
  \end{equation}

\refCNN\ holds with $p_1=1$ and $p_k=0$, $k\neq1$ (\ie, $\doo=1$ \as),
so $\muoo=1$ and $\nuoo=0$. Further, 
$\mun=1+O(\sqrt n)\to\muoo$  
but $\nun=\frac1n d_1(d_1-1)\to a^2\neq\nuoo$ so \refCN{C2} fails.
We have
$1+\muoo^2/(\muoo-\nuoo)=2$, and thus the conclusion of \refT{Tsus}(i)
does not hold. This shows that \refT{Tsus} can fail if \refCN{C2} does
not hold.

If further $a<1$, then
\begin{equation*}
  1+\frac{\mun^2}{\mun-\nun}
=1+\frac1{1-a^2}+o(1)
=\frac{2-a^2}{1-a^2}+o(1).
\end{equation*}
Consequently, \refT{Tsus3} too fails for this example, which shows
that the theorem does not hold in general witout the assumption of
uniform square integrability. Similarly, in the case $a>1$,
$\nun>\mun$ (at least for large $n$), and \refT{Tsus3} would predict
that $\sus(\gndd)\pto\infty$, while \eqref{estar} shows that in fact
it has a finite limit.
\end{example}

\begin{example}\label{E2star}
Modify \refE{Estar} by taking two vertices with high degree, for example
$d_1=d_2\sim a\sqrt n$ 
and $d_i=1$ for $i\ge3$, for some $a>0$. Thus the components of
$\gndd$ are either (when there is no edge 12) two stars of order
$d_1+1\sim a\sqrt n$
plus $n/2+o(n)$ isolated edges,
or (when there is an edge 12) one component of order $2d_1\sim2a\sqrt n$ plus
$n/2+o(n)$ isolated edges. Both events occur with positive limiting
probabilities. (In fact, a simple calculation of the number of
labelled graphs of the two types shows that the probability of an edge
12 converges to $a^2/(a^2+2)$.) Consequently, 
$\sus(\gndd)$ is either $4a^2+2+o(1)$ or $2a^2+2+o(1)$, and
$\sus(\gndd)$ converges (in distribution) to a two-point distribution
and not to a constant. Similarly, 
$\susq(\gndd)$ is either $2+o(1)$ or $a^2+2+o(1)$, and again there is
a limiting two-point distribution.
Consequently, the conclusions of \refT{Tsus}(i) fail for both $\sus$
and $\susq$.

As in \refE{Estar}, \refCNN\ holds with $p_1=1$ 
but \refCN{C2} fails.
\end{example}

In the remaining examples $p_1=0$, and thus \refCC{C1p1} does not hold.

\begin{example}\label{E2}
  The random 2-regular graph $G(n,2)$ is critical, with 
$p_2=1$ and $p_k=0$, $k\neq2$
(\ie, $\doo=2$ \as).
It is well-known that
$G(n,2)$ has \whp\ several large
components with sizes $\Theta_p(n)$, and it follows that 
$\sus(G(n,2)),\susq(G(n,2))\pto\infty$, in accordance with \refT{Tsus}(ii).

Now perturb this example by adding a suitable number of
vertices of degree 4, say  $n_4=\floor{n^{0.9}}$ and $n_2=n-n_4$. 
By ignoring all vertices of degree 2 in \gnddx (contracting their
adjacent edges to a single edge), we obtain a
random 4-regular multigraph, which \whp\ is connected, \cf{} \refR{Rp1=0}).
Hence, \whp\ all vertices of degree 4 belong to 
a single component. Moreover, by splitting each vertex of degree 4
into two vertices of degree 2, considering the (2-regular)
configuration model for 
these vertices, and then recombining the vertices of degree 4,
it is easily seen that \whp\ this is the giant component $\cc1$ and
that it contains 
all  vertices except a small number of cycles 
with a total size $\le n^{0.2}$, say. Thus \whp\ $\susq(G(n,2))\le
n^{0.4-1}$ so $\susq(G(n,2))\pto0$, 
although this example is critical, and so \refT{Tsus}(ii) fails for it.
\end{example}

\begin{example}
  \label{E0}
If $p_0=1$ (\ie, $\doo=0$ a.s.), then \refCC{C1p1} and \ref{C1l} do not hold.
We have $\muoo=0$ and $\qbp=1$ a.s., so $\susoo=1$.
However, $\sus(\gndd)\pto1$ may fail.

For example, let $d_i=3$ for $i\le n_3$ and $d_i=0$ for $n_3<i\le n$,
where we choose $n_3\=2\floor{a\sqrt n}$ for some $a>0$.
Then \refCNB\ hold except for \refCC{C1l},\ref{C1p1}, with $p_0=1$.
\gndd{} consists of $n_0=n-n_3$ isolated vertices together with a
random cubic graph on $n_3$ vertices. The latter is \whp{} connected,
see \refR{Rp1=0}, and thus \whp{}
\begin{equation*}
  \sus\bigpar{\gndd}=\frac1n(n-n_3+n_3^2)
\to1+4a^2>1.
\end{equation*}
Hence, \refT{Tsus} fails for this example.
\end{example}

\newcommand\AAP{\emph{Adv. Appl. Probab.} }
\newcommand\JAP{\emph{J. Appl. Probab.} }
\newcommand\JAMS{\emph{J. \AMS} }
\newcommand\MAMS{\emph{Memoirs \AMS} }
\newcommand\PAMS{\emph{Proc. \AMS} }
\newcommand\TAMS{\emph{Trans. \AMS} }
\newcommand\AnnMS{\emph{Ann. Math. Statist.} }
\newcommand\AnnPr{\emph{Ann. Probab.} }
\newcommand\CPC{\emph{Combin. Probab. Comput.} }
\newcommand\JMAA{\emph{J. Math. Anal. Appl.} }
\newcommand\RSA{\emph{Random Struct. Alg.} }
\newcommand\ZW{\emph{Z. Wahrsch. Verw. Gebiete} }
\newcommand\DMTCS{\jour{Discr. Math. Theor. Comput. Sci.} }

\newcommand\AMS{Amer. Math. Soc.}
\newcommand\Springer{Springer-Verlag}
\newcommand\Wiley{Wiley}

\newcommand\vol{\textbf}
\newcommand\jour{\emph}
\newcommand\book{\emph}
\newcommand\inbook{\emph}
\def\no#1#2,{\unskip#2, no. #1,} 
\newcommand\toappear{\unskip, to appear}

\newcommand\webcite[1]{
\texttt{\def~{{\tiny$\sim$}}#1}\hfill\hfill}
\newcommand\webcitesvante{\webcite{http://www.math.uu.se/~svante/papers/}}
\newcommand\arxiv[1]{\webcite{arXiv:#1.}}

\def\nobibitem#1\par{}

\end{document}